\documentclass[oneside,english]{amsart}
\usepackage[T1]{fontenc}
\usepackage[latin9]{inputenc}
\usepackage{units}
\usepackage{textcomp}
\usepackage{amsthm}
\usepackage{amssymb}
\usepackage{stackrel}

\makeatletter
\numberwithin{equation}{section} \numberwithin{figure}{section}
\theoremstyle{plain}
\newtheorem{thm}{\protect\theoremname}[section]
  \theoremstyle{plain}
  \newtheorem{cor}[thm]{\protect\corollaryname}

 \DeclareMathOperator{\Span}{span}
 \DeclareMathOperator{\ext}{ext}

\makeatother

\usepackage{babel}
  \providecommand{\corollaryname}{Corollary}
\providecommand{\theoremname}{Theorem}

\begin{document}

\title{Rethinking Polyhedrality for Lindenstrauss Spaces }

\author{Emanuele Casini}

\address{Dipartimento di Scienza e Alta Tecnologia, Università dell'Insubria,
via Valleggio 11, 22100 Como, Italy}

\email{emanuele.casini@uninsubria.it}

\author{Enrico Miglierina}

\address{Dipartimento di Discipline Matematiche, Finanza Matematica ed Econometria,
Università Cattolica del Sacro Cuore, via Necchi 9, 20123 Milano,
Italy}

\email{enrico.miglierina@unicatt.it}

\author{\L ukasz Piasecki}

\address{Instytut Matematyki, Uniwersytet Marii Curie-Sk\l odowskiej, Pl.
Marii Curie-Sk\l odowskiej 1, 20-031 Lublin, Poland}

\email{piasecki@hektor.umcs.lublin.pl}
\begin{abstract}
A recent example by the authors (see \cite{Casini Miglierina =000026
Piasecki Ex}) shows that an old result of Zippin about the existence
of an isometric copy of $c$ in a separable Lindenstrauss space is
incorrect. The same example proves that some characterizations of
polyhedral Lindenstrauss spaces, based on the result of Zippin, are
false. The main result of the present paper provides a new
characterization of polyhedrality for the preduals of $\ell_{1}$ and
gives a correct proof for one of the older. Indeed, we prove that
for a space $X$ such that $X^{*}=\ell_{1}$ the following properties
are equivalent:

\begin{itemize}{\small \par}

\item [\rm {(1)}] $X$ is a polyhedral space;{\small \par}

\item [\rm {(2)}] $X$ does not contain an isometric copy of $c$;{\small \par}

\item [\rm {(3)}] $\sup\left\{ x^{*}(x)\,:\, x^{*}\in\mathrm{ext}\left(B_{X^{*}}\right)\setminus D(x)\right\} <1$
for each $x\in S_{X}$, where $D(x)=\left\{ x^{*}\in
S_{X^{*}}:x^{*}(x)=1\right\} .${\small \par}

\end{itemize}{\small \par}

By known theory, from our result follows that a generic
Lindenstrauss space is polyhedral if and only if it does not contain
an isometric copy of $c$. Moreover, a correct version of the result
of Zippin is derived as a corollary of the main result.
\end{abstract}

\keywords{Polyhedral spaces, preduals of $\ell_{1}$, Lindenstrauss
spaces, isometric copy of $c$.}

\maketitle

\section{Introduction}

Let $B_{X}$ ($S_{X}$) denote the closed unit ball (sphere) in a real
Banach space $X$ and $X^{*}$ denote the dual of $X$. If $K$ is a
closed convex subset of $X$, then by $\ext(K)$ we denote the set of
all extreme points of $K$. A closed convex subset $F$ of $B_{X}$ is
called a \emph{face} of $B_{X}$ if for every $x,y\in B_{X}$ and
$\lambda\in(0,1)$ such that $(1-\lambda)x+\lambda y\in F$ we have
$x,y\in F$. A face $F$ of $B_{X}$ is named a \emph{proper face} if
$F\neq B_{X}$. Here $c$ denotes the Banach space of all real
convergent sequences. It is well known that $c^{*}=\ell_{1}$ and the
duality inducing the standard $w^{*}$-topology in $c^{*}$ is given
by:
\[
f(x)=f(1)\lim x(i)+\sum_{i=1}^{+\infty}f(i+1)x(i)
\]
 where $f=(f(1),f(2),\dots)\in\ell_{1}$ and $x=(x(1),x(2),\dots)\in c$.
In the sequel by $\left\{ e_{n}^{*}\right\} _{n=1}^{+\infty}$ we
denote the standard basis of $\ell_{1}$. Let $H\subseteq X$, then we
denote by $\overline{H}$ the norm closure of $H$. If $A$ is a set in
$X^{*}$, then $A'$ denotes the set of all $w^{*}$\emph{-limit
points} of $A:$
\[
A'=\left\{ x^{*}\in X^{*}:x^{*}\in
w^{*}\mathrm{-cl}\left(A\setminus\left\{ x^{*}\right\}
\right)\right\} .
\]
For $x\in S_{X}$, $D(x)$ is defined as
\[
D(x)=\left\{ x^{*}\in S_{X^{*}}:x^{*}(x)=1\right\} .
\]
We recall that $D(x)$ is a $w^{*}$-compact face for every $x\in
S_{X}$ and consequently
$\ext\left(D(x)\right)=D(x)\cap\ext\left(B_{X^{*}}\right)\neq\emptyset$
by the Krein-Milman theorem.

A normed space is \emph{polyhedral} if the unit balls of all its
finite-dimensional subspaces are polytopes. This definition was
originally introduced by Klee in \cite{Klee 1960} and now it can be
seen as the classical definition of polyhedrality. Nevertheless, a
lot of different notions of polyhedrality are appeared in the
literature (see \cite{Durier =000026 Papini} and \cite{Fonf and
Vesely}). The relationships between each of them are studied in
\cite{Durier =000026 Papini,Fonf and Vesely} from the isometric
point of view whereas their isomorphic classification is established
in \cite{Fonf and Vesely}. To our aims it is important to recall two
alternative notions of polyhedrality. By following the notations
used in \cite{Fonf and Vesely} we list the following properties:

\begin{itemize}

\item [\rm {(GM)}]$x^{*}(x)<1$ whenever $x\in S_{X}$ and $x^{*}\in(\ext(B_{X^{*}}))'$
(originally introduced in \cite{Gleit =000026 McGuigan 1972});

\item [\rm {(BD)}]$\sup\left\{ x^{*}(x)\,:\, x^{*}\in\mathrm{ext}\left(B_{X^{*}}\right)\setminus D(x)\right\} <1$
for each $x\in S_{X}$ (originally introduced in \cite{Brosowski
=000026 Deutsch 1974}).

\end{itemize}

In \cite{Durier =000026 Papini}, Theorem 1 is shown (see also
Theorem 1.2 in \cite{Fonf and Vesely}) that for a general Banach
space $X$ we have
\[
{\rm property}\,{\rm (GM)}\,\Rightarrow\,{\rm property}\,{\rm
(BD)}\,\Rightarrow\, X\,{\rm is}\,{\rm a}\,{\rm polyhedral}\,{\rm
space}.
\]
Without additional assumptions none of the implications above can be
reversed. On the other hand both the papers \cite{Durier =000026
Papini} and \cite{Fonf and Vesely} recall that, if $X$ is a
Lindenstrauss space, the three properties are all equivalent (see
also Proposition 6.22 in \cite{Fonf Lindenstrauss =000026 Phelps
2001}). This result essentialy amount to the paper \cite{Gleit
=000026 McGuigan 1972} by Gleit and McGuigan where some structural
properties of polyhedral Lindenstrauss space are studied. For the
convenience of the reader we explicitely recall the theorem by Gleit
and McGuigan.
\begin{thm}
\label{thm:Gleit =000026 McGuigan}\rm{(Theorem 1.2 in \cite{Gleit
=000026 McGuigan 1972})} Let $X$ be a Lindenstrauss space. Then the
following properties are equivalent:

\begin{itemize}

\item [\rm {(1)}] $X$ enjoys property \rm{(GM)};

\item [\rm {(2)}] $X$ is a polyhedral space;

\item [\rm {(3)}] $X$ does not contain an isometric copy of $c$.

\end{itemize}
\end{thm}
As an immediate consequence of the previous result and of the
implications holding for a generic Banach space, we obtain that
properties (1), (2) and (3) in Theorem \ref{thm:Gleit =000026
McGuigan} and property (BD) should be all equivalent whenever we
consider a Lindenstrauss space. The scheme of the proof of Theorem
\ref{thm:Gleit =000026 McGuigan} given in \cite{Gleit =000026
McGuigan 1972} is the following: (1) $\Rightarrow$ (2) $\Rightarrow$
(3) $\Rightarrow$ (1). It is important to remark that the proof of
Theorem \ref{thm:Gleit =000026 McGuigan} is based on the following
result of Lazar in turn.
\begin{thm}
\label{thm:Lazar}\rm{(Theorem 3 in \cite{Lazar1969})} Let $X$ be a
Lindenstrauss space. Then the following properties are equivalent:

\begin{itemize}

\item [\rm {(1)}] $X$ is a polyhedral space;

\item [\rm {(2)}] $X$ does not contain an isometric copy of $c$;

\item [\rm {(3)}] there is not infinite dimensional $w^{*}$- closed
proper faces of $B_{X^{*}}$;

\item [\rm{(4)}] for every Banach spaces $Y\subset Z$ and every
compact operator $T:Y\rightarrow X$ there exists a compact extension
$\tilde{T}:Z\rightarrow X$ with $\left\Vert \tilde{T}\right\Vert
=\left\Vert T\right\Vert $.

\end{itemize}
\end{thm}
The proof of Theorem \ref{thm:Lazar}, given in \cite{Lazar1969},
proceeds as follows: (1) $\Rightarrow$ (2) $\Rightarrow$ (3)
$\Rightarrow$ (4) $\Rightarrow$ (1).

Recently some of the implications of the two quoted above theorems
have been disproved in \cite{Casini Miglierina =000026 Piasecki Ex}
by an example of a suitable space. Indeed, we found an hyperplane
$W$ of $c$ such that $W$ is a polyhedral predual of $\ell_{1}$ with
an infinite dimensional $w^{*}$-closed proper face of $B_{W^{*}}$
and not enjoying property (GM) . Thus, implications (2)
$\Rightarrow$ (3) in Theorem \ref{thm:Lazar} and (3) $\Rightarrow$
(1) in Theorem \ref{thm:Gleit =000026 McGuigan} are false. The
failure of these results are essentially due to an incorrect result
about the presence of an isometric copy of $c$ in a Lindestrauss
space stated in \cite{Zippin1969} (see below the comments after the
proof of Theorem \ref{thm:main result}) and subsequently used by
Lazar to prove the implications (2) $\Rightarrow$ (3) in Theorem
\ref{thm:Lazar}. As a consequence of these faults, the equivalence
between the polyhedrality of $X$ and the lack of an isometric copy
of $c$ in $X$ remains unproved. Let us recall that, in view of the
known results, it is enough to solve this problem in the case of
$\ell_1$-preduals. The main aim of the present paper is to provide a
correct proof of this equivalence by finding a counterpart of
property (GM) in Theorem \ref{thm:Gleit =000026 McGuigan} that works
as intermediate step in the proof that the lack of an isometric copy
of $c$ in a Lindenstrauss space implies the polyhedrality of $X$.
Indeed, under the assumption that $X$ is a predual of $\ell_1$, the
main result of the paper shows that the correct property to
substitute property (GM) is property (BD) mentioned above. This
result completely characterizes polyhedrality of $\ell_1$-preduals
by means of the structural and geometrical properties. It also
provides a correct reformulation (see Corollary \ref{cor:Zippin}) of
the result, stated in \cite{Zippin1969}, about the existence of an
isometric copy of $c$ in a separable Lindenstrauss space. Finally,
it is worth to mention that polyhedrality plays an important role in
the theory of Lindestrauss spaces. Indeed, it is deeply linked with
the norm preserving extension property for compact operators. It is
well known that a Lindenstrauss space is polyhedral if and only if
for every Banach spaces $Y\subset Z$ and every operator
$T:Y\rightarrow X$ with $\dim T(Y)\leq2$ there exists a compact
extension $\tilde{T}:Z\rightarrow X$ with $\left\Vert
\tilde{T}\right\Vert =\left\Vert T\right\Vert $ (combine Theorem 7.9
in \cite{Lindenstrauss 1964} and Theorem 4.7 in \cite{Klee 1959}).
However, the more general extension property stated in Theorem
\ref{thm:Lazar} (see also Proposition 6.23 in \cite{Fonf
Lindenstrauss =000026 Phelps 2001}) now reveals to be unproven since
the chain of implications between $(1)$ and (4) has an interruption
as remarked above.

\section{Characterization of polyhedral $\ell_{1}$-preduals}

This section is devoted to the main result of the paper. In order to
deal with a characterization of polyhedral Lindenstrauss spaces, by
known results (see, e.g. §22 in \cite{Lacey book}), the only
interesting situation to deal with is that of $\ell_{1}$-preduals.
\begin{thm}
\label{thm:main result}Let $X$ be a predual of $\ell_{1}$. The
following properties are equivalent:

\begin{itemize}

\item [\rm {(1)}] $X$ is a polyhedral space;

\item [\rm {(2)}] $X$ does not contain an isometric copy of $c$;

\item [\rm {(3)}] $\sup\left\{ x^{*}(x)\,:\, x^{*}\in\mathrm{ext}\left(B_{X^{*}}\right)\setminus D(x)\right\} <1$
for each $x\in S_{X}$ (property \rm{(BD)}).

\end{itemize}\end{thm}
\begin{proof}
The implication (1) $\Rightarrow$ (2) is straightforward to prove
since it is well known that $c$ is not a polyhedral space. The
implication (3) $\Rightarrow$ (1) is proved in Theorem 1 in
\cite{Durier =000026 Papini} (see also Theorem 1.2 in \cite{Fonf and
Vesely}). Therefore it remains to prove only the implication (2)
$\Rightarrow$ (3).

By contradiction let us suppose that property (3) does not hold.
Therefore there exist a sequence $\left\{ x_{n}^{*}\right\}
_{n=1}^{+\infty}\subset\ext\left(B_{X^{*}}\right)\setminus D(x)$ and
a point $x\in S_{X}$ such that
\[
x_{n}^{*}(x)\underset{_{n\rightarrow+\infty}}{\longrightarrow}1.
\]
Whitout loss of generality we can consider a subsequence $\left\{
x_{n_{k}}^{*}\right\} $ of $\left\{ x_{n}^{*}\right\} $ such that
$x_{n_{k}}^{*}=e_{n_{k}}^{*}$ for every $k\in\mathbb{N}$ and an
element $e^{*}\in X^{*}$ such that
\[
e_{n_{k}}^{*}\stackrel[_{k\rightarrow+\infty}]{^{w^{*}}}{\longrightarrow}e^{*}.
\]
It is easily seen that $e^{*}(x)=1$ and $\left\Vert e^{*}\right\Vert
=1$ and hence $e^{*}\in D(x)$.

Let us recall that the set $D(x)$ is a weak$^{*}$-compact face. By
the Krein-Milman Theorem we have
\[
D(x)=\left\{
\sum_{i\in\Delta}\delta_{i}d_{i}^{*}:\,\sum_{i\in\Delta}\delta_{i}=1;\,\delta_{i}\geq0\,{\rm
for}\,{\rm all}\, i\in\Delta\right\} ,
\]
where $\left\{ d_{i}^{*}\right\}
_{i\in\Delta}=\ext\left(D(x)\right)$ and $\Delta$ is a finite or
countable set of indexes. Let us introduce the set
\[
K=\left\{ e_{n_{k}}^{*}\right\}
_{k=1}^{+\infty}\cup\ext\left(D(x)\right).
\]
We see at once that:

\begin{itemize}

\item [\rm {(i)}] $K$ is a countable subset of $\ext\left(B_{X^{*}}\right)$;

\item [\rm {(ii)}] $K\cap(-K)=\emptyset$;

\item [\rm {(iii)}] the set $Y=\overline{\Span(K)}$ is a weak$^{*}$-closed
set in $X^{*}$ (see Lemma 1 in \cite{Alspach 1992}).

\end{itemize}

We are now in position to apply Theorem 1.1 in \cite{Gasparis 2002}.
By using this result we obtain that there exists a
$w^{*}$-continuous contractive projection $P_{1}$ from $X^{*}$ onto
$Y$.

The task is now to find a $w^{*}$-continuous contractive projection
$P_{2}$ from $Y$ onto a subspace $V$ of $Y$ itself, such that $V$ is
isometric to $c^{*}$ (endowed with its natural $w^{*}$-topology). In
order to define such a projection, it is convenient to describe the
subspace $Y$ as:
\[
Y=\left\{ y^{*}\in X^{*}:\,
y^{*}=\sum_{k=1}^{+\infty}\alpha_{k}e_{n_{k}}^{*}+\sum_{i\in\Delta}\beta_{i}d_{i}^{*}\,,\,\sum_{k=1}^{+\infty}\left|\alpha_{k}\right|<+\infty,\,\sum_{i\in\Delta}\left|\beta_{i}\right|<+\infty\right\}
.
\]
Let us define the linear map $P_{2}:Y\rightarrow X^{*}$ by
\[
P_{2}(y^{*})=\sum_{k=1}^{+\infty}\alpha_{k}e_{n_{k}}^{*}+\left(\sum_{i\in\Delta}\beta_{i}\right)e^{*}.
\]
We remark that

\begin{itemize}

\item [\rm {(iv)}] $P_{2}$ is a projection from $Y$ onto its closed
subspace $V$ defined by
\[
V=\left\{ v\in X^{*}:\,
v=\sum_{k=1}^{+\infty}\alpha_{k}e_{n_{k}}^{*}+\theta e^{*}\mathrm{,
where}\,\sum_{k=1}^{+\infty}\left|\alpha_{k}\right|<+\infty,\,\theta\in\mathbb{R}\right\}
.
\]
Indeed, $e^{*}\in D(x)$. Hence there exists $\left\{
\varepsilon_{i}\right\} _{i\in\Delta}\subset\left[0,1\right]$ such
that $\sum_{i\in\Delta}\varepsilon_{i}=1$ and
$\sum_{i\in\Delta}\varepsilon_{i}d_{i}^{*}=e^{*}.$ Let us consider
$v=\sum_{k=1}^{+\infty}\alpha_{k}e_{n_{k}}^{*}+\theta e^{*}\in V$,
then
\[
P_{2}(v)=P_{2}\left(\sum_{k=1}^{+\infty}\alpha_{k}e_{n_{k}}^{*}+\theta\left(\sum_{i\in\Delta}\varepsilon_{i}d_{i}^{*}\right)\right)=\sum_{k=1}^{+\infty}\alpha_{k}e_{n_{k}}^{*}+\theta\left(\sum_{i\in\Delta}\varepsilon_{i}\right)e^{*}=v.
\]

\item [\rm {(v)}] $V$ is a $w^{*}$-closed subspace of $X^{*}$ (to
prove this fact it is sufficient to remember that
$w^{*}-\lim_{k\rightarrow+\infty}e_{n_{k}}^{*}=e^{*}$);

\item [\rm {(vi)}] $P_{2}$ is a $w^{*}$-continuous map;

\item [\rm {(vii)}] $\left\Vert P_{2}\right\Vert =1$ because $\Delta\cap\left\{ n_{k}\right\} _{k=1}^{+\infty}=\emptyset$.

\end{itemize}

By combining the two projections $P_{1}$ and $P_{2}$ introduced
above we now define the linear map $P:X^{*}\longrightarrow V$ such
that $P(x^{*})=P_{2}\left(P_{1}(x^{*})\right)$ for every $x^{*}\in
X^{*}.$ The map $P$ enjoys the following properties:

\begin{itemize}

\item [\rm {(viii)}] $P$ is a $w^{*}$-continuous and contractive
projection from $X^{*}$ onto $V$;

\end{itemize} the subspace $V$ is isometric to $c^{*}$ (endowed
with its standard $w^{*}$-topology) by means of the
$w^{*}$-continuous isometry $T:V\longrightarrow c^{*}$ defined by
\[
T(e^{*})=f_{1}^{*}\quad\mathrm{and}\quad
T(e_{n_{k}}^{*})=f_{k+1}^{*}\,\mathrm{for}\,\mathrm{all}\, k\geq1,
\]
where $\left\{ f_{n}^{*}\right\} _{n=1}^{+\infty}$ is the standard
basis of $c^{*}=\ell_{1}$.

Let us denote by $^{\perp}V$ the annihilator of $V\subseteq X^{*}$
in $X$; since $P$ and $T$ are $w^{*}$-continuous, there exist two
linear bounded operators
\[
S:\nicefrac{X}{^{\perp}V}\longrightarrow X\quad\mathrm{and}\quad
R:c\longrightarrow\nicefrac{X}{^{\perp}V}
\]
 such that $S^{*}=P$, $R^{*}=T$ and $R$ is an isometry. The linear
map $S\circ R:c\longrightarrow X$ is an isometric injection. This
fact concludes the proof.
\end{proof}
As a consequence of our results, Theorem 2.3 in \cite{Lazar =000026
Lindenstrauss 1971}, Theorem 5 in \cite{Lacey book}, p.226, and
Lemma 1 in \cite{Lacey book}, p.232, we obtain a complete
characterization of Lindenstrauss spaces that are polyhedral.
\begin{cor}
Let $X$ be a Lindenstrauss space. The following properties are
equivalent:

\begin{itemize}

\item [\rm {(1)}]$X$ does not contain an isometric copy of $c$;

\item [\rm {(2)}]$X$ is a polyhedral space.

\end{itemize}
\end{cor}
In \emph{Concluding remarks} of \cite{Zippin1969}, Zippin stated
that a separable Lindenstrauss space $X$ contains a contractively
complemented isometric copy of $c$ whenever the unit ball of $X$ has
at least one extreme point. In \cite{Casini Miglierina =000026
Piasecki Ex} an example, built upon a careful study of the
hyperplanes of $c$, disproves this result. By recalling Proposition
3.1 in \cite{Casini Miglierina =000026 Piasecki FPP} and considering
the previous results, we obtain something more about the presence of
an isometric copy of $c$ in a separable Lindenstrauss space. In
particular, we show the existence of a contractively complemented
copy of $c$. The next result can be seen as a correct version of the
Zippin's result.
\begin{cor}
\label{cor:Zippin}Let $X$ be a separable Lindenstrauss space. The
following properties are equivalent:

\begin{itemize}

\item [\rm {(1)}]$X$ contains a contractively complemented isometric
copy of $c$;

\item [\rm {(2)}]there exists $x\in S_{X}$ such that $\sup\left\{ x^{*}(x)\,:\, x^{*}\in\mathrm{ext}\left(B_{X^{*}}\right)\setminus D(x)\right\} =1$
.

\end{itemize}
\end{cor}
It is worth to underline that the assumption, assumed in
\cite{Zippin1969}, of the mere existence of an extreme point of
$B_{X}$ is far from to be sufficient in order to ensure the
existence of an isometric copy of $c$ in $X$, as shown by the
example considered in \cite{Casini Miglierina =000026 Piasecki Ex}.
Indeed, condition (2) in Corollary \ref{cor:Zippin} gives a much
more detailed description of the geometry of $B_{X^{*}}$ and the
point $x\in S_{X}$ is not necessarily an extreme point (for
instance, let us consider the space $X$ given by $c$ itself).

\end{document}